\documentclass[12pt, a4paper]{amsart}

\usepackage[%hypertex,
unicode=true,
plainpages = false,
pdfpagelabels,
bookmarks=true,
bookmarksnumbered=true,
bookmarksopen=true,
breaklinks=true,
backref=false,
colorlinks=true,
linkcolor = blue,
urlcolor  = blue,
citecolor = red,
anchorcolor = green,
hyperindex = true,
hyperfigures
]{hyperref}
\hypersetup{
	pdftitle={Linear algebra, theory and algorithms},
	pdfauthor={V.H. Mikaelian},
	pdfsubject={Linear Algebra}
}

\usepackage{comment}

\usepackage{enumitem} 
\setlist[itemize]{leftmargin=*}
\setlist[enumerate]{labelsep=*, leftmargin=1.7pc}

\newcommand{\V}{{\mathfrak V}}
\newcommand{\Ni}{{\mathfrak N}}
\newcommand{\B}{{\mathfrak B}}

\newcommand{\X}{{\mathfrak X}}
\newcommand{\Y}{{\mathfrak Y}}

\renewcommand{\U}{{\mathfrak U}}
\newcommand{\A}{{\mathfrak A}}

\newcommand{\var}[1]{\mathrm{var}\left( #1 \right)}
\newcommand{\varr}[1]{\mathrm{var}( #1 )}

\newcommand{\Wr}{\,\mathrm{Wr}\,}
\newcommand{\Wrr}{\,\mathrm{wr}\,}

\usepackage[bitstream-charter]{mathdesign}

\DeclareSymbolFont{cmsymbols}{OMS}{cmsy}{m}{n}
\SetSymbolFont{cmsymbols}{bold}{OMS}{cmsy}{b}{n}
\DeclareSymbolFontAlphabet{\mathcal}{cmsymbols}

\theoremstyle{plain}
\newtheorem{Theorem}{Theorem}[section]
\newtheorem{Lemma}[Theorem]{Lemma}
\newtheorem{Proposition}[Theorem]{Proposition}  
\newtheorem{Corollary}[Theorem]{Corollary}  
  
\theoremstyle{definition}
  
\theoremstyle{remark}
\newtheorem{Remark}[Theorem]{\it Remark} 
\newtheorem{Example}[Theorem]{\it Example} 

\numberwithin{equation}{section}

\begin{document}

{\footnotesize  
\noindent
Journal of Group Theory, ISSN 1435-4446\\
Accepted for publication}

\vskip6mm

%%%%%%%%%%%%%%%%%%%%%%%%%%%%%%%%%%%%%%%%%%%%%%
%%%%%%%%%%%%%%%%%%%%%%%%%%%%%%%%%%%%%%%%%%%%%%


\title{Subvariety structures in certain product varieties of groups}

\author{V. H. Mikaelian
%\\
%Y\lowercase{erevan} S\lowercase{tate} %U\lowercase{niversity}
}

\begin{abstract}
We classify certain cases when the wreath products of distinct pairs of groups  generate the same variety. This allows us to investigate the subvarieties of some nilpotent-by-abelian product varieties  $\U\V$ with the help of wreath products of groups. In particular,  using wreath products we find such subvarieties in nilpotent-by-abelian $\U\V$, which have the same nilpotency class, the same length of solubility, and the same exponent, but which still are distinct subvarieties. Obtained classification strengthens our recent work on varieties generated by wreath products.
\end{abstract}

\date{\today}

\maketitle

%%%%%%%%%%%%%%%%%%%%%%%%%%%%%%%%%%%%%
%%%%%%%%%%%%%%%%%%%%%%%%%%%%%%%%%%%%%
%%%%%%%%%%%%%%%%%%%%%%%%%%%%%%%%%%%%%
\section{\bf Introduction}

\noindent
The objective of this note is to study subvarieties generated by wreath products in certain product varieties of groups, and to discover the cases when wreath products of distinct pairs of groups generate the same variety of groups. 
Or in more specific notation given below, we  investigate the cases when the  equality  \eqref{Equation varieties are equal} holds for the  pairs of groups $A_1, B_1$ and $A_2, B_2$.

Wreath products are among the main tools to study products $\U\V$ of varieties of groups. The methods used in the literature %most 
typically consider groups $A$ and $B$ generating $\U$ and $\V$ respectively, and then find extra conditions, under which the wreath product ${A \Wr B}$ generates $\U\V$, i.e., conditions, under which the equality
\begin{equation}
\tag{$*$} 
\label{EQUATION_main}    
\var{A \Wr B} = \var{A} \var{B}
\end{equation}
holds for $A$ and $B$ (here \textit{Cartesian} wreath product is assumed, but all  results we bring are true for \textit{direct} wreath products also). 
%The advantage of this aproach is that having the equality \eqref{EQUATION_main} we by Birkhoff's Theorem can get all the groups in $\var{A} \var{B}=\U\V$ by applying the operations of taking the homomorphic images, subgroups, Cartesian products to \textit{single} group $A \Wr B$ \cite[15.23]{HannaNeumann}.
%
For chronological development of this approach and for background information on varieties of groups or on wreath products we refer to 
%important work 
\cite{HannaNeumann, 
	3N,
	Some_remarks_on_varieties,
	Baumslag nilp wr,
	B+3N, 
	ShmelkinOnCrossVarieties, Burns65}
and to literature cited therein.

Generalizing some results in the cited literature, we in \cite{AwrB_paper}-\cite{Classification Theorem} were able to suggest criteria classifying all the cases when \eqref{EQUATION_main} holds for groups from certain classes of groups: abelian groups, $p$-groups, nilpotent groups of finite exponent, etc.~(see, in particular, very brief outline of results in Section 5 of \cite{Classification Theorem}).

\smallskip
Here we turn to a sharper problem of comparison of two varieties, both generated by wreath products% (not necessarily correlated with product varieties)
. Namely, take $A_1, B_1$ and $A_2, B_2$ to be pairs of non-trivial groups such that\,
$\varr {A_1}=\varr {A_2}$, \, $\varr {B_1}=\varr {B_2}$,\, and distinguish the cases when the equality
\begin{equation}
\tag{$**$}
\label{Equation varieties are equal}
\var{A_1 \!\Wr B_1} = \var{A_2 \!\Wr B_2}
\end{equation}
holds. 
The main classification criterion 
given in Theorem~\ref{Theorem var(Wr) are equal} 
covers the cases of nilpotent $A_1,A_2$ and abelian $B_1,B_2$, with some restrictions on exponents.

%If we take $A_2$ and $B_2$ so that $\var{A_2 \Wr B_2}$ generates the product $\var{A_2} \var{B_2}$, then \eqref{Equation varieties are equal} is reduced to \eqref{EQUATION_main}.
Besides getting a generalization of \eqref{EQUATION_main} our study of equality \eqref{Equation varieties are equal} is motivated by some applications one of which we would like to outline here.
Classification of subvariety structures of $\U\V$ is incomplete even when $\U$ and $\V$ are such ``small'' varieties as the abelian varieties $\A_m$ and $\A_n$ respectively. Here are some of the  results in this direction:
$\A_p$ (for prime numbers $p$) are the simplest non-trivial varieties,  as they consist of the Cartesian powers of the cycle $C_p$ only. L.G.~Kov\'acs and M.F.~Newman in~\cite{KovacsAndNewmanOnNon-Cross} fully described the subvariety structure in the product $\A_p^2 = \A_p \A_p$ for $p>2$. 
Later they continued this classification for the varieties $ \A_{p^u} \A_p$. Their research was unpublished for many years, and it appeared in 1994 only \cite{Kovacs Newman Ravello} (parts of their proof are present in \cite{Bryce_Metabelian_groups}).
Another direction is  description of subvarieties in the product $\A_m \A_n$ where $m$ and $n$ are \textit{coprime}. This is done by C.~Houghton (mentioned by Hanna Neumann in \cite[54.42]{HannaNeumann}), by
P.\,J.~Cossey
(Ph.D. thesis \cite{Cossey66}, mentioned by R.A.~Bryce in \cite{Bryce_Metabelian_groups}).
A more general result of R.A.~Bryce classifies the subvarieties of $\A_m \A_n$, where $m$ and $n$ are \textit{nearly prime} in the sense that, if a prime $p$ divides $m$, then $p^2$ does not divide $n$ \cite{Bryce_Metabelian_groups}.
In 1967 Hanna Neumann wrote that classification of subvarieties of  $\A_m \A_n$ for arbitrary $m$ and $n$  \textit{``seems within
	reach''} \cite{HannaNeumann}. 
And R.A.~Bryce in 1970 mentioned that \textit{``classifying all metabelian varieties
	is at present slight'' } \cite{Bryce_Metabelian_groups}. However, nearly  half a century later this task is not yet accomplished: Yu.A.~Bakhturin and A.Yu.~Olshanskii remarked in the  survey \cite{Bakhturin Olshanskii} of 1988 and of 1991 that \textit{``classification of all nilpotent metabelian group varieties has not been completed yet''.}

\smallskip
As this brief summary shows, one of the cases, when the subvariety structure of $\U\V$ is less known, is the case when $\U$ and $\V$ have non-coprime exponents divisible by high powers $p^u$ for many prime numbers $p$.
Thus, even if we cannot classify all the subvarieties in some product varieties $\U\V$, it may be  interesting to find those subvarieties in $\U\V$, which are generated by wreath products. %(especially in cases when exponents of $A$ and $B$ have many common prime divisors).
We, surely, can take any groups $A \in \U$ and $B \in \V$, and then $A \Wr B$ will generate some subvariety in $\U\V$. 
But in order to make this approach reasonable, we yet have to detect \textit{if or not two wreath products of that type generate the same subvariety}, i.e, if or not the equality \eqref{Equation varieties are equal} holds for the given pairs of groups.

\smallskip
Yet another outcome of this research may be stressed.
In the literature the different subvarieties are often distinguished by their different nilpotency classes, different lengths of solubility, or different exponents (see, for example, classification of subvarieties of $\A_p^2$ in~\cite{KovacsAndNewmanOnNon-Cross}).
Using wreath products technique, we below construct such subvarieties of $\U\V$, which have the same nilpotency class, the same length of solubility, the same exponent, but which still are distinct subvarieties (see examples in Section~\ref{SE Examples} below).

\smallskip
Since our study of \eqref{EQUATION_main} was in detail presented in publications, some of which are very recent,  we do not want to directly or indirectly  repeat here any proof fragment which might
have been presented in our earlier publications. Instead, we just refer to facts in respective articles, and give the links to our ArXiv files  in References.

%%%%%%%%%%%%%%%%%%%%%%%%%%%%%%%%%%%%%
%%%%%%%%%%%%%%%%%%%%%%%%%%%%%%%%%%%%%
%%%%%%%%%%%%%%%%%%%%%%%%%%%%%%%%%%%%%
\section{\bf Equivalence of the $p$-primary components and the main theorem}
\label{SE Equivalence and  main theorem}

\noindent
Before we turn to the main consideration evolving groups of \textit{finite exponents} only, let us briefly discuss the equality \eqref{Equation varieties are equal} for some other cases also. 

If $B_1$ and $B_2$ both are abelian groups \textit{not} of finite exponent, then they are \textit{discriminating} groups~\cite{HannaNeumann, 3N,  AwrB_paper} and, thus, for any $A_1$ and $A_2$ the equality \eqref{EQUATION_main} holds for the pairs $A_1, B_1$ and $A_2, B_2$   \cite{3N, HannaNeumann}. 
If we in addition have the conditions
$\varr {A_1}=\varr {A_2}$ and $\varr {B_1}=\varr {B_2}$,
then clearly $\varr {A_1}\, \varr {B_1}=\varr {A_2} \,\varr {B_2}$, and so \eqref{Equation varieties are equal} also takes place for these groups $A_1, B_1, \; A_2, B_2$.

If one of $B_1$ and $B_2$ is of finite exponent and the other is not, then  $\varr {B_1}\neq \varr {B_2}$, and \eqref{Equation varieties are equal} not to be considered for this case, at all. 

This examples show why the case of finite exponents is the most interesting one, and starting from here we are going to consider that case only.

\smallskip
By Pr{\"u}per's Theorem any abelian group $B$ of finite exponent is a direct product of its finite cyclic subgroups. Recall that a $p$-primary component $B(p)$ of $B$ is the subgroup of all elements of $B$, whose orders are powers of $p$. And since $B$ clearly is a direct product of its $p$-primary components $B(p)$, the orders of the mentioned cyclic subgroups can be supposed to be powers of primes. %
If the cardinality of cyclic factors $C_{p^u}$ of order $p^u$ in this decomposion is $\mathfrak m_{p^{u}}$, we can write their direct product as $C_{p^{u}}^{\mathfrak m_{p^{u}}}$\!\!. Then $B(p)$ is a product of some summands of that type:
\begin{equation}
\label{Equation decomposition of B(p)}
B(p)
=C_{p^{u_1}}^{\mathfrak m_{p^{u_1}}}
\!\!\times \cdots \times
C_{p^{u_r}}^{\mathfrak m_{p^{u_r}}}\!\!\!,
\end{equation}
where we may suppose $u_1 \ge \cdots \ge u_r$. The cardinal numbers $\mathfrak m_{p^{u_1}},\ldots,\mathfrak m_{p^{u_r}}$ are invariants of $B(p)$ in the sense that they characterize $B(p)$ uniquely (see L.~Fuchs' textbook \cite[Section 35]{Fuchs1}, from where we also adopted the symbols $\mathfrak m_{p^{u}}$ and the above notation). 
If $B(p)$ is finite, then all the cardinals $\mathfrak m_{p^{u_1}},\ldots,\mathfrak m_{p^{u_r}}$ clearly are finite. Otherwise, at least one of them will be infinite, and we can always choose the \textit{first} one of such infinite invariants $\mathfrak m_{p^{u_i}}$.
%
%[\textbf{Add $p$-primary component definition, definition of $C_p$,  definition of $C_p^n$}]

\begin{Example}
	\label{Example decomposition samples}
Consider the group
$
B=C_{3^5}^6 \times C_{3^3}^{\aleph_0} \times C_{3^2}^5 \times C_{3}^{\aleph}
\,\,\times\,\, 
C_{5^3}^4 \times C_{5^2}.
$
For the $3$-component $B(3)$ of $B$ we have
$u_1=5$, ${\mathfrak m_{3^{u_1}}}\!=6$;\;
$u_2=3$, ${\mathfrak m_{3^{u_2}}}=\aleph_0$;\;
$u_3=2$, ${\mathfrak m_{3^{u_3}}}=5$;\;
$u_4=1$, ${\mathfrak m_{3^{u_4}}}=\aleph$. So, the \textit{first} infinite factor of $B(3)$ is $C_{3^3}^{\aleph_0}$, although the factor $C_{3}^{\aleph}$ is of higher cardinality.
And for the $5$-component $B(5)$  we have 
$u_1=3$, ${\mathfrak m_{5^{u_1}}}=4$;\;
$u_2=2$, ${\mathfrak m_{5^{u_2}}}=1$, so $B(5)$ has \textit{no} infinite factor.
\end{Example}

Let $B_1$ and $B_2$ be abelian groups of finite exponent, and let for a prime $p$ their $p$-primary components $B_1(p)$ and $B_2(p)$ have direct decompositions of the above type:
\begin{equation}
\label{Equation decomposition of B_1(p)}
B_1(p)
=C_{p^{u_1}}^{\mathfrak m_{p^{u_1}}}
\!\!\times \cdots \times
C_{p^{u_r}}^{\mathfrak m_{p^{u_r}}}\!\!\!,
\end{equation}
\begin{equation}
\label{Equation decomposition of B_2(p)}
B_2(p)
=C_{p^{v_1}}^{\mathfrak m_{p^{v_1}}}
\!\!\times \cdots \times
C_{p^{v_s}}^{\mathfrak m_{p^{v_s}}}\!\!.
\end{equation}
Define a specific equivalence relation $\equiv$ between such $B_1(p)$ and $B_2(p)$.
Namely:
\begin{enumerate}
	\item if $B_1(p),B_2(p)$ both are \textit{finite}, then  $B_1(p) \equiv B_2(p)$ if and only if $B_1(p)$ and $B_2(p)$ are isomorphic, i.e., $r=s$ and $u_i=v_i$, ${\mathfrak m_{p^{u_i}}}={\mathfrak m_{p^{v_i}}}$ for each $i=1,\ldots,r$;
	\item if $B_1(p),B_2(p)$ both are \textit{infinite}, then  $B_1(p) \equiv B_2(p)$ if and only if there is a $k$ such that:
\begin{enumerate}
\item[\textit{i}) ] $u_i=v_i$ and ${\mathfrak m_{p^{u_i}}}={\mathfrak m_{p^{v_i}}}$ for each $i=1,\ldots,k-1$;
\item[\textit{ii}) ] $C_{p^{u_k}}^{\mathfrak m_{p^{u_k}}}$ is the first infinite factor in decomposition \eqref{Equation decomposition of B_1(p)}, $C_{p^{v_k}}^{\mathfrak m_{p^{v_k}}}$ is the first infinite factor in decomposition \eqref{Equation decomposition of B_2(p)}, and $u_k=v_k$;
\end{enumerate}
	\item $B_1(p),B_2(p)$ are \textit{not} equivalent for all other cases.
\end{enumerate}

Notice that in point (\textit{ii}) above we do not require isomorphism of $C_{p^{u_k}}^{\mathfrak m_{p^{u_k}}}$ and $C_{p^{v_k}}^{\mathfrak m_{p^{v_k}}}$\!\!. We require them both to be direct products of \textit{infinitely} many copies of the same cyclic group $C_{p^{u_k}}$. In particular, $\mathfrak m_{p^{u_k}}$ and $\mathfrak m_{p^{v_k}}$ can be distinct infinite cardinal numbers.

\begin{Example} 
	\label{Example equivalent samples}
In order to get a group equivalent to the group  $B$ of Example~\ref{Example decomposition samples}, we can replace in $B$ the factors $C_{3^2}^5$  and $C_{3}^{\aleph}$ by arbitrary direct product of copies of the cycles $C_{3^2}$ and $C_3$. Also, we can replace $C_{3^3}^{\aleph_0}$ by, say, $C_{3^3}^{\aleph}$. However, we cannot alter any of the remaining factors $C_{3^5}^6$, $C_{5^3}^4$, $C_{5^2}$.
\end{Example}

In these terms our main theorem is:

\begin{Theorem}
	\label{Theorem var(Wr) are equal}%%%%%%%%%%%%%%%%%%%%%%%%%%%%%%%%%%%%%%%%%%%%%
	Let $A_1, A_2$ be non-trivial nilpotent groups of  exponent $m$ generating the same variety, and let $B_1,\,B_2$ be non-trivial abelian groups of exponent $n$  generating the same variety, where any prime divisor $p$ of $n$ also divides $m$. Then equality \eqref{Equation varieties are equal} holds for $A_1, A_2, B_1, B_2$ if and only if $B_1(p) \equiv B_2(p)$ for each $p$.
\end{Theorem}

Notice how the roles of the passive and active groups of these wreath products are different: for $A_1, A_2$ we just require that $\var{A_1} = \var{A_2}$, whereas for $B_1, B_2$ we put extra conditions on
structures of their decompositions.

\begin{Corollary}
	\label{Corollary finite groups}%%%%%%%%%%%%%%%%%%%%%%%%%%%%%%%%%%%%%%%%%%%%%
In the notations of Theorem~\ref{Theorem var(Wr) are equal}:
\begin{enumerate}
	\item[\rm (a)] equality  \eqref{Equation varieties are equal} holds for finite groups $B_1,B_2$  if and only if $B_1$ and $B_2$ are isomorphic;
	\item[\rm (b)] equality  \eqref{Equation varieties are equal} never holds if one of the groups $B_1,B_2$ is finite, and the other is infinite.
\end{enumerate}
\end{Corollary}

\begin{comment}
\begin{Theorem}
	%\label{Theorem var(Wr) are equal}%%%%%%%%%%%%%%%%%%%%%%%%%%%%%%%%%%%%%%%%%%%%%
	OLD	OLD	OLD	OLD	OLD	OLD	Let $A$ be a non-trivial nilpotent group of  exponent $m$, and let $B_1,B_2$ be non-trivial abelian groups of exponent $n \mathrel{|} m$. Then equality \eqref{Equation varieties are equal} holds for $A, B_1, B_2$ if and only if $B_1(p) \equiv B_2(p)$ for each prime divisor $p$ of $n$.
\end{Theorem}
\end{comment}

%%%%%%%%%%%%%%%%%%%%%%%%%%%%%%%%%%%%%
%%%%%%%%%%%%%%%%%%%%%%%%%%%%%%%%%%%%%
%%%%%%%%%%%%%%%%%%%%%%%%%%%%%%%%%%%%%
\section{\bf The proof of Theorem~\ref{Theorem var(Wr) are equal}}

\noindent
Since we are going to intensively use nilpotency classes of wreath products, let us proceed to introduction of D.~Shield's formula~\cite{Shield nilpotent a, Shield nilpotent b}.
By the well know theorem of G.~Baumslag, a Cartesian wreath product $A \Wr B$ of non-trivial groups $A$ and $B$ is nilpotent if and only if $A$ is a nilpotent $p$-group of finite exponent, and $B$ is a finite $p$-group for a prime number $p$~\cite{Baumslag nilp wr}. The analog of this also holds for direct wreath products.
H.~Liebeck calculated the nilpotency class of $A \Wr B$ for the particular case when the groups  $A$ and $B$ of G.~Baumslag's theorem are abelian~\cite{Liebeck_Nilpotency_classes}.
The complete formula for \textit{any} nilpotent $p$-group $A$ of finite exponent and  \textit{any} finite $p$-group $B$ is found by D.~Shield, and to present it we need to briefly introduce the $K_p$-series. 
See also \cite{K_p-series} where we explain the construction in more details, with illustrative examples. 

For an (arbitrary) group $B$ and the prime number $p$ the $K_p$-series $K_{i,p}(B)$ of $B$ is defined for $i=1, 2, \ldots$ as the product 
$$ 
K_{i,p}(B) = 
\prod \left\{
\gamma_r(B)^{p^j} 
\mathrel{|}
\text{for all $r, j$ such that $r p^j \ge i$}
\right\}\!,
$$
where $\gamma_r(B)$ is the $r$'th term of the lower central series of the group $B$, and $\gamma_r(B)^{p^j}=\langle
g^{p^j}
\mathrel{|}
g\in \gamma_r(B)
\rangle$ is the $p^j$'th power of $\gamma_r(B)$.

Clearly, $K_{1,p}(B) = B$ holds for any $B$. 
If $B$ is abelian, then $\gamma_2(B)=[B,B]=\{1\}$, and so 
in the product above the powers $\gamma_1(B)^{p^j}=B^{p^j}$ of the initial term $\gamma_1(B) = B$ need to be considered only. 
%

%\begin{Example}
%	If $G=C_{p^3} \times C_{p} \times C_{p}$ with $p = 5$, then it is easy to calculate that:
%	$$
%	K_{1,p}(G) = G; 
%	\quad
%	K_{2,p}(G) = \cdots = K_{5,p}(G) \cong C_{p^2}; 
%	$$
%	$$
%	K_{6,p}(G) = \cdots = K_{25,p}(G) \cong C_{p};
%	\quad
%	K_{26,p}(G) \cong \{1\}.
%	$$
%\end{Example}

When $B$ is a finite $p$-group, the following additional parameters are introduced:
let $d$ be the maximal integer such that $K_{d,p}(B) \not= \{1\}$. For each $s=1,\ldots, d$ define $e(s)$ from equality
$
p^{e(s)} = |K_{s,p} / K_{s+1,p}|,
$
and set $a$ and $b$ by the rules:
$
a = 1 + (p-1) \sum_{s=1}^d \left(s \cdot e(s)\vphantom{a^b}\right),
$
$
b = (p-1)d.
$

Finally, let $s(h)$ be defined as follows: $p^{s(h)}$ is the exponent of the $h$'th term $\gamma_h(A)$ of the lower central series of the nilpotent $p$-group $A$ of finite exponent.

Then by Shield's formula \cite{Shield nilpotent b}, the nilpotency class of the wreath product $A \Wr B$ is equal to the maximum
\begin{equation}
\label{EQUATION A wr B class}    
\max_{h = 1, \ldots, \, c} \{
a \, h + (s(h)-1)b
\}.
\end{equation}

\vskip5mm
After these preparations turn back to the  proof for Theorem~\ref{Theorem var(Wr) are equal}.
The first and very simple step to start with is to reduce  equality \eqref{Equation varieties are equal} to its particular case when $A_1 = A_2=A$:
\begin{equation}
\label{Equation varieties are equal ONE A}
\var{A \Wr B_1} = \var{A \Wr B_2}.
\end{equation}

Recall that for any given class $\X$ of groups ${\sf Q}\X$, ${\sf S}\X$, ${\sf C}\X$ denote the classes of all homomorphic images, subgroups, Cartesian products  of groups of $\X$ respectively. By Birkhoff's Theorem~\cite[15.23]{HannaNeumann}, for any class  of groups $\X$ the equality  $\varr{\X}={\sf QSC}\,\X$ holds.

\begin{Lemma}
	\label{LE reduce to case A1=A2}%%%%%%%%%%%%%%%%%%%%%%%%%%%%%%%%%%%%%%%%%%%%%
	For any groups $A_1, A_2$ generating the same variety the equality
	\begin{equation}
	\label{EQ reduce to case A1=A2}
	\var{A_1 \Wr B} = \var{A_2 \Wr B}
	\end{equation}
	holds for any group $B$.
\end{Lemma}
\begin{proof} 
	If $A_1 \in \var {A_2}$, then $A_1 \in{\sf QSC}\,\{A_2\}$. 
	Then by 
	\cite[22.11]{HannaNeumann} and
	\cite[Lemma 1.1]{AwrB_paper} we have 
	${A_1 \Wr B} \in \var{A_2 \Wr B}$ for any $B$. The inverse incusion ${A_2 \Wr B} \in \var{A_1 \Wr B}$ is proved analogously.
\end{proof} 

Turning back to notations of Theorem~\ref{Theorem var(Wr) are equal}, we always have the equality $\var{A_1 \Wr B_1} = \var{A_2 \Wr B_1}$ for the groups $A_1, A_2, B_1$. So, we can just denote $A_1= A_2=A$, and reduce considration of the equality \eqref{Equation varieties are equal} to study of \eqref{Equation varieties are equal ONE A}, which is going to be our main objective for the sequel. 
We are going to achieve it in the following steps: first we consider the case of $p$-groups only, and find a few necessary conditions for the equality \eqref{Equation varieties are equal ONE A} for some specific cases of $p$-groups in
Lemma~\ref{Lemma case of finie initial factors},  
Lemma~\ref{Lemma case of finite end infinite factrs}. Then 
Lemma~\ref{Lemma sufficiency for p-groups}
shows that the combination of the collected necessary conditions also is sufficient.
In
Lemma~\ref{Lemma p-groups belong} and in the final proof
we will deduce the general case from the cases obtained for $p$-groups. 

\medskip
Suppose $A, B_1,B_2$ are non-trivial nilpotent $p$-group of finite exponent, and $B_1,B_2$ have the decompositions \eqref{Equation decomposition of B_1(p)} and \eqref{Equation decomposition of B_2(p)} respectively (in which we clearly have $B_1(p)=B_1$, $B_2(p)=B_2$, since we deal with $p$-groups). 

For any groups $X$ and $Y$ of finite exponents we have $\exp{(X \Wr Y)}=\exp{(X)} \cdot \exp{(Y)}$.
Thus, the first easy observation is that, if \eqref{Equation varieties are equal ONE A} holds for $A, B_1,B_2$, then the exponents of $B_1$ and $B_2$ are equal, i.e., $u_1=v_1$. Otherwise, the wreath products  $A \Wr B_1$ and $A \Wr B_2$ would also have distinct exponents, and they would generate distinct varieties. 

If $\mathfrak m_{p^{u_1}}$ and $\mathfrak m_{p^{v_1}}$ are finite and equal, then the first factors $C_{p^{u_1}}^{\mathfrak m_{p^{u_1}}}$ and $C_{p^{v_1}}^{\mathfrak m_{p^{v_1}}}$ in \eqref{Equation decomposition of B_1(p)} and \eqref{Equation decomposition of B_2(p)} are coinciding finite groups. 
For the sequel reserve the letter $t$ to denote the index, for which 
$C_{p^{u_i}}^{\mathfrak m_{p^{u_i}}}$ and $C_{p^{v_i}}^{\mathfrak m_{p^{v_i}}}$ 
in \eqref{Equation decomposition of B_1(p)} and \eqref{Equation decomposition of B_2(p)} 
are coinciding finite factors for each $i=1,\ldots,t-1$, but \textit{not} for $i=t$. 
That is, the $t$'th factors $C_{p^{u_t}}^{\mathfrak m_{p^{u_t}}}$ and $C_{p^{v_t}}^{\mathfrak m_{p^{v_t}}}$ either are non-isomorphic finite groups, or at least one of them is infinite.
Clearly, if $B_1\ncong B_2$, then such a $t$ exists, and the case $t=1$ is not ruled out. 

\begin{Lemma}
	\label{Lemma case of finie initial factors} %%%%%%%%%%%%%%%%%%%%%%%%%%%%%
{\color{red} } In the above circumstances, if the $t$'th factors $C_{p^{u_t}}^{\mathfrak m_{p^{u_t}}}$\! and $C_{p^{v_t}}^{\mathfrak m_{p^{v_t}}}$  are non-isomorphic finite groups, then the equality \eqref{Equation varieties are equal ONE A} does not hold for $A,\,B_1,B_2$.

\end{Lemma}

\begin{proof} We have two options:
	\begin{enumerate}
		\item[\rm (a)] either $u_t = v_t$, and then ${\mathfrak m_{p^{u_t}}} \neq {\mathfrak m_{p^{v_t}}}$, and so we may suppose  ${\mathfrak m_{p^{u_t}}} > {\mathfrak m_{p^{v_t}}}$;
		\item[\rm (b)] or $u_t \neq v_t$, and we may suppose $u_t > v_t$ (the values of ${\mathfrak m_{p^{u_t}}}$ and ${\mathfrak m_{p^{v_t}}}$ are immaterial in this case). 	 
	\end{enumerate}	
	Let us conduct the proof for the first option, and the second can be proved by slight adaptation, by adding to $B_2$ the ``missing'' factor $C_{p^{v_t}}^{\mathfrak m_{p^{v_t}}}$ with $v_t=u_t$ and ${\mathfrak m_{p^{v_t}}}\!\!=0$.	
	
	We suppose the equation \eqref{Equation varieties are equal ONE A} holds, and proceed to arrive to a contradiction. 	
	Denote $u_t = v_t= w$, and set  $P_1=\left(A \Wr B_1 \right)^{p^{w-1}}$ \!\!and $P_2=\left(A \Wr B_2 \right)^{p^{w-1}}$\!\!\!\!\!. \,
	Since the wreath product $A \Wr B_2$ is an extension of its base subgroup $A^{B_2}$ by the active group $B_2$, then 
	$$
	P_2 / (A^{B_2} \cap P_2) \cong (P_2\, A^{B_2}) / A^{B_2} \le
	(A \Wr B_2) / A^{B_2} \cong B_2,
	$$
	i.e., the subgroup $P_2$ of $A \Wr B_2$ is an extension of some subgroup $A^{B_2} \cap P_2$ of $A^{B_2}$ by some subgroup of $B_2$ and, in fact, of $B_2^{p^{w-1}}$\! taking into account the multiplication rule in wreath products. By the Kaloujnine-Krassner theorem~\cite[22.21]{HannaNeumann} $P_2$ can be embedded in the wreath product $A^{B_2} \Wr B_2^{p^{w-1}}$\!\!\!\!\!,\;\; and we can find an \textit{upper} bound for the nilpotency class of $P_2$ by applying the Shield's formula to this wreath product.
	Clearly, 
	\begin{equation}
	\label{EQUATION B_2^{p^{w-1}}} B_2^{p^{w-1}} \!\! \cong C_{p^{v_1-w+1}}^{\mathfrak m_{p^{v_1}}}
	\!\!\times \cdots \times
	C_{p^{v_{t-1}-w+1}}^{\mathfrak m_{p^{v_{t-1}}}}
	\times
	C_{p}^{\mathfrak m_{p^{v_t}}}\!\!\!\!,
	\end{equation}
	and it is just a simple routine computation to find the $K_p$-series for \eqref{EQUATION B_2^{p^{w-1}}}: 
	
	$$K_{1,p}(B_2^{p^{w-1}})=B_2^{p^{w-1}}\!\!\!\!\!;$$
	$$K_{2,p}(B_2^{p^{w-1}}) = \cdots = K_{p,p}(B_2^{p^{w-1}})=
	\left(B_2^{p^{w-1}}\right)^p$$
	because for the indices $i=2,\ldots, p$ the least power $p^j$ for which $p^j\ge i$ is achieved for $j=1$. 
	Clearly,
	$\left(B_2^{p^{w-1}}\right)^p \cong C_{p^{v_1-w}}^{\mathfrak m_{p^{v_1}}}
	\!\!\times \cdots \times
	C_{p^{v_{t-1}-w}}^{\mathfrak m_{p^{v_{t-1}}}}
	$ (notice how the factor $C_{p}^{\mathfrak m_{p^{v_t}}}$\! disappeared as it is of exponent $p$).
	$$K_{p+1,p}(B_2^{p^{w-1}}) = \cdots = K_{p^2,p}(B_2^{p^{w-1}})=
	\left(B_2^{p^{w-1}}\right)^{p^2}$$
	because for the indices $i=p+1,\ldots, p^2$ the least power $p^j$ for which $p^j\ge i$ is achieved for $j=2$. 
	So,  
	$\left(B_2^{p^{w-1}}\right)^{p^2} \!\!\cong C_{p^{v_1-w-1}}^{\mathfrak m_{p^{v_1}}}
	\!\times \cdots \times
	C_{p^{v_{t-1}-w-1}}^{\mathfrak m_{p^{v_{t-1}}}}
	$.

	As the process continues some more and more factors will be lost. And in the last steps we for some $r$ get:
	$$K_{p^r\,+\;1,p}(B_2^{p^{w-1}}) = \cdots = K_{p^{r+1},p}(B_2^{p^{w-1}})=
	\left(B_2^{p^{w-1}}\right)^{p^{r+1}}\!\!\!\! =
	C_{p}^{\mathfrak m_{p^{v_1}}}\!\!\!. $$
	And, finally, $$K_{p^{r+1}\;+\;1,p}(B_2^{p^{w-1}})=K_{p^{r+1}\;+\;2,p}(B_2^{p^{w-1}})=\cdots=\{1\}.$$
	In these notations $d=p^{r+1}$\!. 
	The values of $e(s)$, $s=1,\ldots, d$ are easy to get as
	they are non-zero in the following cases only:
	$$
	|K_{1,p}(B_2^{p^{w-1}})/K_{2,p}(B_2^{p^{w-1}})|=p^{{\mathfrak m_{p^{v_1}}} +\cdots + {\mathfrak m_{p^{v_t}}} }=p^{e(1)},
	$$
	$$
	|K_{p,p}(B_2^{p^{w-1}})/K_{p+1,p}(B_2^{p^{w-1}})|=p^{e(p)},
	$$
	$$
	|K_{p^2,p}(B_2^{p^{w-1}})/K_{p^2+1,p}(B_2^{p^{w-1}})|=p^{e(p^2)},
	$$
	$$
	\cdots \cdots \cdots \cdots \cdots \cdots \cdots \cdots \cdots \cdots \cdots \cdots \cdots \cdots 
	$$
	$$
	|K_{p^{r+1},p}(B_2^{p^{w-1}})/K_{p^{r+1}+1}(B_2^{p^{w-1}})|
	=|K_{d,p}(B_2^{p^{w-1}})/K_{d+1}(B_2^{p^{w-1}})|=p^{e(p^{r+1})}=p^{e(d)}$$
	(by our agreement all these $e(s)$, $s=1,\ldots, d$ are finite).
	
	It remains to compute the values of $a$, $b$, $s(h)$, and to get the nilpotency class of $A^{B_2} \Wr B_2^{p^{w-1}}$ by formula \eqref{EQUATION A wr B class}. We do \textit{not} need write down that class because information relevant for this proof already is collected above. 
	\medskip
	
	Next we construct a specific subgroup $P_1^*$ in $P_1^{\hphantom{*}}$. Let $\{a_l \mathrel{|} l\in \mathcal L\}$ be any generating set of $A$, and choose a set of generators  for the first $t$ factors of $B_1$ in \eqref{Equation decomposition of B_1(p)} as follows. For each $i=1,\ldots,t$ let $c_{ij}$ be a generator of the $j$'th copy of the cycle $C_{p^{u_i}}$ in $C_{p^{u_i}}^{\mathfrak m_{p^{u_i}}}$ for $j = 1,\ldots, {\mathfrak m_{p^{u_i}}}$, i.e.,  $C_{p^{u_i}}^{\mathfrak m_{p^{u_i}}} = \prod_{j=1}^{{\mathfrak m_{p^{u_i}}}} \langle c_{ij} \rangle$. We get ${\mathfrak m}={\mathfrak m_{p^{u_1}}} +\cdots + {\mathfrak m_{p^{u_t}}}$ generators $c_{ij}$ in total,
and each element $b$ in the product of the first $t$ factors of $B_1$ can be written as
\begin{equation}
\label{EQ presentation of b}
b = c_{11}^{\varepsilon_1}\cdots c_{t\, {\mathfrak m_{p^{u_t}}}}^{\varepsilon_{{\mathfrak m}}}
\end{equation}
for some non-negative integer exponents $\varepsilon_1,\ldots,\varepsilon_{{\mathfrak m}}$.
	For each generator $a_l$ define an element $\pi_l$ in the base subgroup $A^{B_1}$ of $A \Wr B_1$ as follows. 
Set 
$\pi_l\,(b)=a_l$, if 	
in presentation \eqref{EQ presentation of b} of $b$ the first exponent	
$\varepsilon_1$ is zero, and all the remaining exponents $\varepsilon_{2},\ldots,\varepsilon_{{\mathfrak m}}$ are less than $p^{w-1}$\!.
Set $\pi_l\,(b)=1$ for all other values of  $b\in B_1$.

Using standard arguments (see, for example, Section 5 in \cite{Embedding theorems for groups 59}) 
	it is easy to verify that
$\left(c_{11}\pi_l\right)^{p^{w-1}}\!\!\!=\left(c_{11}\right)^{p^{w-1}}\!\theta_l$ where
$\theta_l\,(b)=a_l$, if 	
in presentation \eqref{EQ presentation of b}  all the  exponents $\varepsilon_{1},\ldots,\varepsilon_{{\mathfrak m}}$ are less than $p^{w-1}$\hskip-5mm,\hskip5mm and 
 $\theta_l\,(b)=1$ for  other values   $b\in B_1$.
Since $\theta_l=
	\left(c_{11}\right)^{-p^{w-1}}  \! \! 
	\left(c_{11}\pi_l\right)^{p^{w-1}}$\hskip-5mm,\hskip5mm we have $\theta_l \in P_1=\left(A \Wr B_1 \right)^{p^{w-1}}$\hskip-5mm,\hskip5mm
	and elements $\theta_l$ together with all the powers $c_{ij}^{p^{w-1}}$ are generating a subgroup $P_1^*$  in $P_1$. Each element $c_{ij}^{p^{w-1}}$\! is generating a cyclic subgroup $\langle c_{ij}^{p^{w-1}} \rangle$ of order $p^{u_i-w+1}$ inside the respective cyclic subgroup $\langle c_{ij} \rangle \cong C_{p^{u_i}}$.
Taking into account the ``$p^{w-1}$ steps shifting effect'' of the elements $c_{i'j'}^{p^{w-1}}$ on the base subgroup for any indices $i',j'$,\; it is easy to see that  $P_1^*$  is isomorphic to the wreath product:
	\begin{equation}
	\label{EQ new wreath product}
	A \Wr B_1^{p^{w-1}}  \!\cong\, A \Wr 
	\left( C_{p^{u_1-w+1}}^{\mathfrak m_{p^{u_1}}}
	\!\!\times \cdots \times
	C_{p^{u_{t-1}-w+1}}^{\mathfrak m_{p^{u_{t-1}}}}
	\times
	C_{p}^{\mathfrak m_{p^{u_t}}}\!
	\right)
	=
	A \Wr B_1^*
	.
	\end{equation}
	If we now by Shield's formula compute the nilpotency class of \eqref{EQ new wreath product}, it will also be a \textit{lower} bound for the nilpotency class of $P_1$.
	
	The $K_p$-series for the active group $B_1^*$ is: 
	$$K_{1,p}(B_1^*)=B_1^*,$$
	$$K_{2,p}(B_1^*) = \cdots = K_{p,p}(B_1^*)=
	\left(B_1^*\right)^p \cong C_{p^{u_1-w}}^{\mathfrak m_{p^{u_1}}}
	\!\!\times \cdots \times
	C_{p^{u_{t-1}-w}}^{\mathfrak m_{p^{u_{t-1}}}}$$
	(the factor $C_{p}^{\mathfrak m_{p^{u_t}}}$ again disappeared as it is of exponent $p$).
	$$K_{p+1,p}(B_1^*) = \cdots = K_{p^2,p}(B_1^*)=
	\left(B_1^*\right)^{p^2} \cong C_{p^{u_1-w-1}}^{\mathfrak m_{p^{u_1}}}
	\!\!\times \cdots \times
	C_{p^{u_{t-1}-w-1}}^{\mathfrak m_{p^{u_{t-1}}}}.$$
	And in the last steps we for the same $r$ get:
	$$K_{p^r+1,p}(B_1^*) = \cdots = K_{p^{r+1},p}(B_1^*)=
	\left(B_1^*\right)^{p^{r+1}}\!\!\!\! =
	C_{p}^{\mathfrak m_{p^{u_1}}}.$$
	And, finally, $K_{p^{r+1}+1,p}(B_1^*)=K_{p^{r+1}+2,p}(B_1^*)=\cdots=\{1\}$.
	
	We again have the same $d=p^{r+1}$. 
	The only non-zero values of $e(s)$ for  $s=1,\ldots, d$ are:
	$$
	|K_{1,p}(B_1^*)/K_{2,p}(B_1^*)|=p^{{\mathfrak m_{p^{u_1}}} +\cdots + {\mathfrak m_{p^{u_t}}} }=p^{e(1)},
	$$
	$$
	|K_{p,p}(B_1^*)/K_{p+1,p}(B_1^*)|=p^{e(p)},
	$$
	$$
	|K_{p^2,p}(B_1^*)/K_{p^2+1,p}(B_1^*)|=p^{e(p^2)},
	$$
	$$
	\cdots \cdots \cdots \cdots \cdots \cdots \cdots \cdots \cdots \cdots \cdots \cdots \cdots \cdots 
	$$
	$$
	|K_{p^{r+1},p}(B_1^*)/K_{p^{r+1}+1}(B_1^*)|
	=|K_{d,p}(B_1^*)/K_{d+1}(B_1^*)|=p^{e(p^{r+1})}=p^{e(d)}.
	$$
	Then we can compute the values of $a$, $b$, $s(h)$, and get the nilpotency class of $W^*$ by formula \eqref{EQUATION A wr B class}. 
	
	Let us compare the parameters $e(s)$, $a$, $b$, $s(h)$ that we above calculated for the wreath products $A^{B_2} \Wr B_2^{p^{w-1}}$ \!\!\!and  $A \Wr B_1^*$ respectively.
	The parameter $e(1)$ is larger for the wreath product  $A \Wr B_1^*$ than for $A^{B_2} \Wr B_2^{p^{w-1}}$ because ${\mathfrak m_{p^{u_t}}} > {\mathfrak m_{p^{v_t}}}$, and so the direct factors $C_p$ appear in $B_1^*$ strictly more times than in $B_2^{p^{w-1}}$\!\!\!\!\!.\;\;
	The parameters $e(2), e(3), \ldots , e(d)$ will be the same for $A^{B_2} \Wr B_2^{p^{w-1}}$\! and for $A \Wr B_1^*$ because $u_i=v_i$ for all $i=1,\ldots,t-1$.
	
	Since $a$ is calculated as $a = 1 + (p-1) \sum_{s=1}^d \left(s \cdot e(s)\vphantom{a^b}\right)$, we have that the value of the parameter $a$ is strictly larger for $A \Wr B_1^*$ rather than for $A^{B_2} \Wr B_2^{p^{w-1}}$\!\!\!\!\!.\;\;
	Since $b = (p-1)d$, this parameter is the same for both wreath products.
	
	It remains to compare the parameters $s(h)$. For $A^{B_2} \Wr B_2^{p^{w-1}}$\!\! the value of $s(h)$ is set so that $p^{s(h)}$ is the exponent of the $h$'th term $\gamma_h(A^{B_2})$ of $A^{B_2}$.
	And for $A \Wr B_1^*$ the value of $s(h)$ is set so that $p^{s(h)}$ is the exponent of  $\gamma_h(A)$ of $A$. Since these exponents clearly are equal, the parameters $s(h)$ also are the same for both wreath products.
	
	We get that the only value that is different in Shield's formula applied to two wreath products is $e(1)$, and it is strictly larger for $A \Wr B_1^*$. Thus, the nilpotency class of $A \Wr B_1^*$ is strictly larger than that of $A^{B_2} \Wr B_2^{p^{w-1}}$\!\!\!.
	In other words, a lower bound for the class of $P_1$ is larger than an upper bound for $P_2$. 
	
	If $c_2$ is the nilpotency classes of $P_2$, then 
	$P_2=\left(A \Wr B_2 \right)^{p^{w-1}}$\!\! is in nilpotent variety $ \Ni_{c_2}$, and thus, the group $A \Wr B_2$ (together with the variety it generates) is in product  $\Ni_{c_2} \B_{p^{w-1}}$ of $ \Ni_{c_2}$ and of Burnside variety $\B_{p^{w-1}}$. But as we saw $A \Wr B_1$ does not belong to this product. Thus, also $\varr{A \Wr B_1} \neq \varr{A \Wr B_2}$. Contradiction.
\end{proof}

\begin{Remark} 
	With some more routine in proofs we could show that the upper and lower bounds found above, in fact, exactly are the nilpotency classes of $A \Wr B_1$ and $A \Wr B_2$ respectively. We, however, refrain from doing that, as the proof above already accomplishes the task we needed.\;
Also notice that even if $B_1$ and $B_2$ after their initial $t$ finite factors contained some \textit{infinite} factors $C_{p^{u_i}}^{\mathfrak m_{p^{v_i}}}$ or $C_{p^{v_i}}^{\mathfrak m_{p^{u_i}}}$\! for $i>t$, they played no role in Lemma~\ref{Lemma case of finie initial factors}, because all such factors disappeared in the $(p^{w-1})$'th powers $B_1^{p^{w-1}}$\!\!\! and $B_2^{p^{w-1}}$\!\!\!\!.
\end{Remark}

We can already deduce:

\begin{Proposition}
	\label{Corollary finite p-groups}%%%%%%%%%%%%%%%%%%%%%%%%%%%%%%%%%%%%%%%%%%%%%
If $A$ is a non-trivial nilpotent $p$-group of finite exponent, and $B_1,B_2$ are non-trivial finite abelian  $p$-groups, then the equality \eqref{Equation varieties are equal ONE A} holds for $A,\,B_1,B_2$ if and only if $B_1 \cong B_2$.
\end{Proposition}

\begin{proof}
Sufficiency of the condition is evident as $B_1 \cong B_2$ implies $A \Wr B_1 \cong A \Wr B_2$. And necessity follows from the above lemma.
\end{proof}

It is time to allow one or both of the $t$'th factors in $B_1$ or $B_2$ to be \textit{infinite}, preserving all other conditions and \textit{excluding} the case covered by previous lemma (when the $t$'th factors are finite non-isomorphic groups).
We are not ruling out the option $t=1$, i.e., the groups $B_1$ or $B_2$ may start by an infinite factor.

\begin{Lemma}
	\label{Lemma case of finite end infinite factrs} %%%%%%%%%%%%%%%%%%%%%%%%%%%%%
In the above circumstances, if the $t$'th factors $C_{p^{u_t}}^{\mathfrak m_{p^{u_t}}}$\!\! and $C_{p^{v_t}}^{\mathfrak m_{p^{v_t}}}$  are not both infinite groups of the same exponent, then the equality \eqref{Equation varieties are equal ONE A} does not hold for $A,\,B_1,B_2$.
\end{Lemma}

\begin{proof} 
Without loss of generality we have the following cases of which the first three cover the situations when only one of the $t$'th factors is infinite:

\smallskip
\noindent
\textit{Case 1.} ${\mathfrak m_{p^{u_t}}}$ is infinite, ${\mathfrak m_{p^{v_t}}}$ is finite, and $u_t = v_t =w$ are equal. This time we cannot apply the Shield's formula to $P_1=\left(A \Wr B_1 \right)^{p^{w-1}}$\! because $\left(B_1 \right)^{p^{w-1}}$\! is infinite. 
Consider a new group $B_1'$ which is obtained from $B_1$ by replacing its $t$'th factor $C_{p^{u_t}}^{\mathfrak m_{p^{u_t}}}$ by $C_{p^{u_t}}^{\mathfrak m_{p^{v_t} }+ 1}$\!\!\!\!\!\!\!,\,\,\,\,\,\, i.e., in its direct decomposition we replace the infinitely many copies of  $C_{p^{u_t}}$ by finitely many ${\mathfrak m_{p^{v_t} }\!\!+ \!1}$ copies of the same cycle. The Shield's formula can be applied to the group $\left(A \Wr B_1' \right)^{p^{w-1}}$\!\!\!\!\!\!\!,\,\,\, and by the proof to Lemma~\ref{Lemma case of finie initial factors} we see that its nilpotency class is higher then the class $c_2$ of $P_2$. Thus, as in previous proof, $A \Wr B_1'$ does not belong to the product variety $\Ni_{c_2} \B_{p^{w-1}}$ which contains $A \Wr B_2$.
But $B_1'\le B_1$ and so by~\cite[22.13]{HannaNeumann} or by \cite[Lemma 1.2]{AwrB_paper} $A \Wr B_1'$ is isomorphic to a subgroup of $A \Wr B_1$. So, $A \Wr B_1$ also is not in $\Ni_{c_2} \B_{p^{w-1}}$.

\smallskip
\noindent
\textit{Case 2.} 
${\mathfrak m_{p^{u_t}}}$ is infinite, ${\mathfrak m_{p^{v_t}}}$ is finite, and $u_t > v_t$. We can add in $B_2$ (right before the finite $t$'th factor) a new factor $C_{p^{v_t}}^{\mathfrak m_{p^{v_t}}}$\!\! with $v_t=u_t$ and ${\mathfrak m_{p^{v_t}}}=0$.  We are in a situation already covered by Case 1. 

\smallskip
\noindent
\textit{Case 3.} 
${\mathfrak m_{p^{u_t}}}$ is infinite, ${\mathfrak m_{p^{v_t}}}$ is finite, and $u_t < v_t$. We can add in $B_1$ (right before the infinite factor) a new factor $C_{p^{u_t}}^{\mathfrak m_{p^{u_t}}}$\!\! with $u_t=v_t$ and ${\mathfrak m_{p^{u_t}}}=0$. %The old first infinite factor will now be written as $C_{p^{u_{t+1}}}^{\mathfrak m_{p^{u_{t+1}}}}$\!\!. 
We are in a situation already covered by Lemma~\ref{Lemma case of finie initial factors}.

\smallskip
\noindent
\textit{Case 4.} 
Both ${\mathfrak m_{p^{u_t}}}$ and ${\mathfrak m_{p^{v_t}}}$ are infinite, and $u_t \neq v_t$. 
We can reduce this to one of the previous cases by adding one more finite factor to one of the groups $B_1$ or $B_2$.

\smallskip
There only remains the case when both ${\mathfrak m_{p^{u_t}}}$ and ${\mathfrak m_{p^{v_t}}}$ are finite, and it is ruled out above in the lemma.
\end{proof} 

The series of necessary conditions restricted our consideration to the situation where $B_1 \equiv B_2$, i.e.,
$B_1$ and $B_2$ are of the same exponent; in their decompositions \eqref{Equation decomposition of B_1(p)} and \eqref{Equation decomposition of B_2(p)} the  initial $t-1$ finite factors are the same; $B_1$ and $B_2$ still may differ in their $t$'th factors, and in such a case both $C_{p^{u_t}}^{\mathfrak m_{p^{u_t}}}$ and $C_{p^{v_t}}^{\mathfrak m_{p^{v_t}}}$\!\! are infinite and have the same exponent  
(these two factors need not to be isomorphic, as ${\mathfrak m_{p^{u_t}}}$ and ${\mathfrak m_{p^{v_t}}}$ may be distinct infinite cardinal numbers).
%
%If such infinite factors are not present in  \eqref{Equation decomposition of B_1(p)} and \eqref{Equation decomposition of B_2(p)}, then $B_1$ and $B_2$ simply are isomorphic finite abelian groups. Observe that both for finite and infinite $B_1,B_2$  we have $B_1 \equiv B_2$.
%
Two special cases are not ruled out: we may have $t=1$ (i.e., the initial coinciding finite factors are absent in $B_1$ and $B_2$); or  $t-1=r=s$, i.e., $B_1$ and $B_2$ are finite isomorphic groups.

\begin{Lemma}
	\label{Lemma sufficiency for p-groups} %%%%%%%%%%%%%%%%%%%%%%%%%%%%%
In the above circumstances, with equivalence $B_1 \equiv B_2$, the equality \eqref{Equation varieties are equal ONE A}  holds for $A,\,B_1,B_2$.
\end{Lemma}

\begin{proof} 
If $B_1,B_2$ are finite, then  $B_1 \cong B_2$ holds together with $A \Wr B_1 \cong A \Wr B_2$, and we are left nothing to prove.

Suppose $B_1,B_2$ are infinite, their first coinciding $t-1$ factors are finite, and their $t$'th factors both are infinite and are of the same exponent. Let us first prove an auxiliary fact.
Fix any infinite cardinal number ${\mathfrak s}$, and denote 
$B_{{\mathfrak s}} = C_{p^{u_1}}^{\mathfrak m_{p^{u_1}}}
\!\!\times \cdots 
\times
C_{p^{u_{t-1}}}^{\mathfrak m_{p^{u_{t-1}}}}
\!\!\times
C_{p^{u_t}}^{{\mathfrak s}}$, i.e., $B_{{\mathfrak s}}$ can be obtained from $B_1$ (or from $B_2$) by taking its first $t-1$ finite factors and by adding one more infinite direct factor $C_{p^{u_t}}^{{\mathfrak s}}$. 
In particular, when ${\mathfrak s}=\aleph_0$, we get the group $B_{\aleph_0}$.
The following equality holds:
\begin{equation}
\label{Equation var(A wr B_aleph_0)=var(A wr B_n)}
\varr{A \Wr B_{\aleph_0}}=
\varr{A \Wr B_{{\mathfrak s}}}.
\end{equation}
Indeed, since $\aleph_0 \le {\mathfrak s}$, $B_{\aleph_0}$ is a subgroup of $B_{{\mathfrak s}}$ and, by \cite[22.13]{HannaNeumann}  or by \cite[Lemma 1.2]{AwrB_paper} $A \Wr B_{\aleph_0}$ is a subgroup of $A \Wr B_{{\mathfrak s}}$, and so $\varr{A \Wr B_{\aleph_0}} \subseteq \varr{A \Wr B_{{\mathfrak s}}}$. 
If the inverse inclusion does \textit{not} hold, then there is a word $v(x_1,\ldots,x_k)$ in the absolutely free group $F_k$, such that $v(x_1,\ldots,x_k) \equiv 1$ is an identity for the group $A \Wr B_{\aleph_0}$ but \textit{not} for the larger group $A \Wr B_{{\mathfrak s}}$, that is, there are some elements $g_1,\ldots,g_k \in A \Wr B_{{\mathfrak s}}$ such that $v(g_1,\ldots,g_k) \neq 1$.

Each $g_j$, $j=1,\ldots,k$, is of a form $g_j=b_j \theta_j$ where $b_j \in B_{{\mathfrak s}}$ and $\theta_j \in A^{B_{{\mathfrak s}}}$.
For any groups their Cartesian wreath product and direct wreath product always generate the same variety of groups \cite[22.31]{HannaNeumann}. Thus,  we may suppose $g_1,\ldots,g_k$ already are in the  direct wreath product, that is, the values $\theta_j(d)$  are non-trivial for at most finitely many elements $d_1,\ldots,d_l\in B_{{\mathfrak s}}$.
Moreover, since the product
$B_{{\mathfrak s}} = C_{p^{u_1}}^{\mathfrak m_{p^{u_1}}}
\!\!\times \cdots 
\times
C_{p^{u_{t-1}}}^{\mathfrak m_{p^{u_{t-1}}}}
\!\!\times
C_{p^{u_t}}^{{\mathfrak s}}$
also is direct,
the evolved elements $d_1,\ldots,d_l$ and $b_1,\ldots,b_k$ inside this direct product have at most finitely many non-trivial coordinates in respective copies of the cycles in $C_{p^{u_{1}}},\ldots,C_{p^{u_{t-1}}}, C_{p^{u_{t}}}$ in $B_{{\mathfrak s}}$.
In particular, only finitely many coordinates are taken in the factor $C_{p^{u_t}}^{{\mathfrak s}}$. Clearly, keeping some countably many copies of $C_{p^{u_t}}$, and dropping all the remaining copies of that cycle we will get the product $C_{p^{u_t}}^{\aleph_0}$, and $v(g_1,\ldots,g_k) \neq 1$ will still hold in 
$B_{\aleph_0}$.  Contradiction.

Turning to the main proof first notice that $\aleph_0 \le {\mathfrak m_{p^{u_t}}}$ and $\aleph_0 \le {\mathfrak m_{p^{v_t}}}$ and, thus, $B_{\aleph_0}$ is a subgroup both in $B_1$ and $B_2$. So, we have:
\begin{equation}
\label{Equation var(A wr B^aleph) is less}
\varr{A \Wr B_{\aleph_0}} \subseteq
\varr{A \Wr B_{i}}, \quad
i=1,2.
\end{equation}
Next denote by $\overline B_1$ the group obtained from $B_1$ by replacing in its decomposition \eqref{Equation decomposition of B_1(p)} all the factors
$C_{p^{u_j}}^{\mathfrak m_{p^{u_j}}}$ by  
$C_{p^{u_t}}^{\mathfrak m_{p^{u_t}}}$ for all $j=t+1,\ldots,r$.
We similarly define the group $\overline B_2$. It is clear that:
$$
C_{p^{u_t}}^{\mathfrak m_{p^{u_t}}}
\!\!\times 
C_{p^{u_t}}^{\mathfrak m_{p^{u_{t+1}}}}
\!\!\times\cdots \times
C_{p^{u_t}}^{\mathfrak m_{p^{u_r}}}
\!\!\!=
C_{p^{u_t}}^{\mathfrak s_{1}}
,\quad\quad
%%%%%
C_{p^{v_t}}^{\mathfrak m_{p^{v_t}}}
\!\!\times 
C_{p^{v_t}}^{\mathfrak m_{p^{v_{t+1}}}}
\!\!\times\cdots \times
C_{p^{v_t}}^{\mathfrak m_{p^{v_s}}}
\!\!=
C_{p^{v_t}}^{\mathfrak s_{2}}
\!=
C_{p^{u_t}}^{\mathfrak s_{2}},
$$
with 
$\mathfrak s_{1}=\max\{
{\mathfrak m_{p^{u_{t}}}},
{\mathfrak m_{p^{u_{t+1}}}},\ldots,{\mathfrak m_{p^{u_r}}}
\}$ 
and 
$\mathfrak s_{2}=\max\{
{\mathfrak m_{p^{v_{t}}}},
{\mathfrak m_{p^{v_{t+1}}}},\ldots,{\mathfrak m_{p^{v_s}}}
\}$
(recall that $p^{u_t}=p^{v_t}$). 
We see that $\overline B_i$ in fact is $B_{{\mathfrak s}_i}$, $i=1,2$.
On the other hand $B_i$ clearly is a subgroup in $\overline B_i$, and so:
\begin{equation}
\label{Equation var(A wr B^n1) is less}
\varr{A \Wr B_{i}} \subseteq
\varr{A \Wr B_{{\mathfrak s}_i}}, \quad
i=1,2.
\end{equation}
Let now ${\mathfrak s}$ be the maximum of ${\mathfrak s}_1$ and ${\mathfrak s}_2$.
By equalities
\eqref{Equation var(A wr B_aleph_0)=var(A wr B_n)},
\eqref{Equation var(A wr B^aleph) is less} and 
\eqref{Equation var(A wr B^n1) is less}
we have: 
$$
\varr{A \Wr B_{\aleph_0}} = 
\varr{A \Wr B_{i}} = 
\varr{A \Wr B_{{\mathfrak s}_i}}=
\varr{A \Wr B_{{\mathfrak s}}}
, \quad
i=1,2.
$$
\end{proof}

The lemmas~\ref{Lemma case of finie initial factors}--\ref{Lemma sufficiency for p-groups} already prove the restricted version of Theorem~\ref{Theorem var(Wr) are equal} for $p$-groups:

\begin{Proposition}
	\label{Proposition theorem for p-groups}%%%%%%%%%%%%%%%%%%%%%%%%%%%%%%%%%%%%%%%%%%%%%
	Let $A$ be a non-trivial nilpotent $p$-group of finite exponent, and let $B_1,B_2$ be non-trivial abelian $p$-groups of finite exponents. Then equality \eqref{Equation varieties are equal ONE A} holds for $A, B_1, B_2$ if and only if $B_1 \equiv B_2$.
\end{Proposition}

Starting by the agreement made after Lemma~\ref{LE reduce to case A1=A2}  our consideration assumed that $A$, $B_1$ and $B_2$ are $p$-groups only. From now on let  $A$ be a non-trivial nilpotent group of class $c$ and of exponent $m$, and let $B_1,B_2$ be non-trivial abelian groups of exponent $n$ such that $n \mathrel{|} m$. 
%
\begin{comment}
According to Birkhoff's Theorem~\cite[15.23, 15.31]{HannaNeumann} for any class  of groups $\X$ the variety $\varr{\X}$ is nothing but the class ${\sf QSC}\,\X$.
%
We denote $\X \Wr \Y = \{ X\Wr Y \,|\, X\in \X, Y\in \Y\}$. 
%
The following two lemmas group a few statements we need (see Proposition 22.11 and Proposition 22.13 in \cite{HannaNeumann}, Lemma 1.1 and Lemma 1.2 in \cite{AwrB_paper}) !!!!!!!!!!

we need below:

\begin{Lemma}
	\label{X*WrY_belongs_var}%%%%%%%%%%%%%%%%%%%%%%%%%%%%%%%%%%%%%%%%%%%%%
	For arbitrary classes $\X$ and $\Y$ of groups and for arbitrary groups $X^*$ and $Y$, where either $X^*\in 
	{\sf Q}\X$, or $X^*\in {\sf S}\X$, or $X^*\in {\sf C}\X$, and where $Y\in \Y$, the group $X^* \Wr
	Y$  belongs to the variety $\var{\X \Wr \Y}$.
\end{Lemma}

\begin{Lemma}
	\label{XWrY*_belongs_var}%%%%%%%%%%%%%%%%%%%%%%%%%%%%%%%%%%%%%%%%%%%%%
	For arbitrary classes $\X$ and $\Y$ of groups and for arbitrary groups $X$ and $Y^*$, where $X\in\X$ and 
	where $Y^*\in {\sf S}\Y$, the group $X \Wr Y^*$  belongs to the variety $\var{\X \Wr \Y}$.
	Moreover, if  $\X$ is a class of abelian groups, then for each  $Y^*\in {\sf Q}\Y$ the group $X
	\Wrr Y^*$  also belongs to $\var{\X \Wr \Y}$.
\end{Lemma}
\end{comment}
%
Denote $\U =\var{A}$, and assume $p_1, \ldots, p_l$ are all the prime divisors of $m$.
Since $\U$ is nilpotent of class $c$, by  \cite[Corollary 35.12]{HannaNeumann} $\U$ is generated by $F = F_c(\U)$. Being a finite nilpotent group $F$ is a direct product of its Sylow $p_i$-subgroups, $i=1,\ldots,l$:
\begin{equation}
\label{EQUATION_sylow_factors}    
F = S_{p_1} \times \cdots \times S_{p_l},
\end{equation}
and all primes $p_i$ do participate: none of $S_{p_i}$ is trivial, as $\exp(F) = \exp(\U) = \exp(A)$.
Assume $p$ is a prime divisor of  $n$ (and of $m$), i.e., it is one of the primes $p_i$. 
Denote by $p^u$ the highest power of $p$ dividing $n$ (in terms of \eqref{Equation decomposition of B_1(p)} we could write $u=u_1$). 
The following technical fact was proved in \cite{Classification Theorem}:

\begin{Lemma}[see the proof of Lemma 3.3 in \cite{Classification Theorem}]
	\label{Lemma p-groups belong}%%%%%%%%%%%%%%%%%	
Any $p$-group in variety $\var{A \Wr B_1}$ belongs to $\var{S_{p} \! \Wr B_1(p)}$.
\end{Lemma}

The analog of this lemma holds for $p$-groups in variety $\var{A \Wr B_2}$ also.

\smallskip
In order to prove the necessity part of Theorem~\ref{Theorem var(Wr) are equal} (supposing by our agreement made after Lemma~\ref{LE reduce to case A1=A2} that $A_1=A_2=A$), assume the equality \eqref{Equation varieties are equal ONE A} holds for $A, B_1, B_2$. Fix any $p$ dividing $n$, and let $P$ be any group from the variety $\var{S_{p} \! \Wr B_1(p)}$. 
Since the group $F$ and, thus, also $S_{p}$ are in $\U$,  they can be obtained from $A$ using the operations ${\sf Q,S,C}$. Thus, by \cite[Lemma 1.1]{AwrB_paper} and
\cite[Lemma 1.2]{AwrB_paper} (see also \cite[22.11]{HannaNeumann}, \cite[22.13]{HannaNeumann}), $P$ belongs to $\varr{A \Wr B_1}$. By assumption, the latter is equal to $\varr{A \Wr B_2}$. Since $P$ clearly is a $p$-group, and it is in $\varr{A \Wr B_2}$, the group $P$ also belongs to $\var{S_{p} \! \Wr B_2(p)}$ by Lemma~\ref{Lemma p-groups belong}. In the same way we can show that any group from $\var{S_{p} \! \Wr B_2(p)}$ also is in $\var{S_{p} \! \Wr B_1(p)}$, and so:
\begin{equation}
\label{EQUATION_equal subvarietes of p-groups}    
	\var{S_{p} \! \Wr B_2(p)} = \var{S_{p} \! \Wr B_1(p)}.
\end{equation}
Thus, according to Proposition~\ref{Proposition theorem for p-groups}, we get $B_1(p) \equiv B_2(p)$ for any $p \mathrel{|} n$.

\smallskip

To prove the sufficiency part of Theorem~\ref{Theorem var(Wr) are equal} suppose $B_1(p_i) \equiv B_2(p_i)$ for any of the prime divisors $p_1,\ldots,p_h$ of $n$, which means that the decompositions \eqref{Equation decomposition of B_1(p)} and \eqref{Equation decomposition of B_2(p)} for groups $B_1(p_i)$ and $B_2(p_i)$ both start by some  coinciding \textit{finite} factors
$C_{p_i^{u_1}}^{\mathfrak m_{p_i^{u_1}}}
\!\!\times \cdots 
\times
C_{p_i^{u_{t_i-1}}}^{\mathfrak m_{p_i^{u_{t_i-1}}}}$ and 
$C_{p_i^{v_1}}^{\mathfrak m_{p_i^{v_1}}}
\!\!\times \cdots 
\times
C_{p_i^{v_{t_i-1}}}^{\mathfrak m_{p_i^{v_{t_i-1}}}}$
perhaps followed by  \textit{infinite} factors
$C_{p_i^{u_{t_i}}}^{\mathfrak m_{p_i^{u_{t_i}}}}$\! and  $C_{p_i^{v_{t_i}}}^{\mathfrak m_{p_i^{v_{t_i}}}}$\! respectively (the exponents $p_i^{u_{t_i}}$ and $p_i^{v_{t_i}}$ are equal, and ${\mathfrak m_{p_i^{u_{t_i}}}}$ and ${\mathfrak m_{p_i^{v_{t_i}}}}$ may be any infinite cardinals, and the case ${t_i}=1$ is not ruled out).

Now we are going to apply the idea and the notation from the proof of Lemma~\ref{Lemma sufficiency for p-groups}. 
Using the notation $B_{{\mathfrak s}}$ introduced for any cardinal number ${\mathfrak s}$ define 
$$
B_{{\mathfrak s}}^*=\prod_{i=1}^h B_1(p_i)_{{\mathfrak s}}
=\prod_{i=1}^h B_2(p_i)_{{\mathfrak s}},
$$
i.e., for each $p_i$ we take the first \textit{finite} factors from the  decomposition of $B_1(p_i)$, and add one more factor $C_{p_i^{u_{t_i}}}^{\mathfrak s}$
(since $B_1(p_i) \equiv B_2(p_i)$, we clearly have $B_1(p_i)_{{\mathfrak s}} \equiv B_2(p_i)_{{\mathfrak s}}$).

It is clear that taking ${\mathfrak s}=\aleph_0$ we get the group  $B_{\aleph_0}^*$ which is contained as a subgroup in both $B_1$ and $B_2$ (since an infinite ${\mathfrak m_{p_i^{v_{t_i}}}}$ is greater than or equal to $\aleph_0$ for any $i$).
On the other hand, taking ${\mathfrak s}$ to be the maximum of all cardinals ${\mathfrak m_{p_i^{v_{t_i+1}}}},\; {\mathfrak m_{p_i^{v_{t_i+2}}}},\ldots$ for all $i$ we get the group $B_{{\mathfrak s}}^*$ which contains isomorphic copies of both $B_1$ and $B_2$. By \cite[22.13]{HannaNeumann} or by \cite[Lemma 1.2]{AwrB_paper}, we have
$$
A \Wr B_{\aleph_0}^*
\le
A \Wr B_1
\le
A \Wr B_{{\mathfrak s}}^*,
\quad\quad
A \Wr B_{\aleph_0}^*
\le
A \Wr B_2
\le
A \Wr B_{{\mathfrak s}}^*.
$$
A repetition of arguments about the identity $v(x_1,\ldots,x_k) \equiv 1$ in the proof of Lemma~\ref{Lemma sufficiency for p-groups} shows that the wreath products $A \Wr B_{\aleph_0}^*$ and $A \Wr B_{{\mathfrak s}}^*$ generate the same variety. Thus, also $A \Wr B_1$ and $A \Wr B_2$ generate the same variety of groups.

The proof of Theorem~\ref{Theorem var(Wr) are equal} is completed.

%
%%%%%%%%%%%%%%%%%%%%%%%%%%%%%%%%%%%%%
%%%%%%%%%%%%%%%%%%%%%%%%%%%%%%%%%%%%%
\section{\bf Examples}
\label{SE Examples}

\begin{Example} 
	\label{Example the powers have distinct classes}
Take
	$A_1=A_2=A=C_3$,\, $B_1=C_{3^2}^2$,\, $B_2=C_{3^2}\times C_3^4$. For the group $B_1$ we have 
	$K_{1,3}(B_1)=B_1=C_{3^2}^2$,\;
	$K_{2,3}(B_1)=K_{3,3}(B_1)=B_1^{3^1}=C_{3}^2$,\;
	$K_{4,3}(B_1)=K_{5,3}(B_1)=\cdots =\{1\}$.\;
	This means $d=3$;\; 
	$e(1)=2$, since $|K_{1,3}(B_1) / K_{2,3}(B_1)|=3^2$,\;
	$e(2)=0$ since $|K_{2,3}(B_1) / K_{3,3}(B_1)|=1=3^0$,\;
	$e(3)=2$ since $|K_{3,3}(B_1) / K_{4,3}(B_1)|=3^2$.\;
	So, we have $a=1+2(1\cdot 2 + 2\cdot 0 + 3 \cdot 2)=17$,\;
	$b=(3-1)3=6$.\;
	Also $h=1$ (the nilpotency class of $A$) and the exponent of $\gamma_1(A)=A$ is $3^1$, and we have $s(1)=1$. Thus, the nilpotency class of $A \Wr B_1$ is $17\cdot 1 + (1-1)6=17$.
	
	Now do the same starting by $B_2$. We have 
	$K_{1,3}(B_2)=B_2=C_{3^2}\times C_3^4$,\;
	$K_{2,3}(B_2)=K_{3,3}(B_2)=B_2^{3^1}=C_{3}$,\;
	$K_{4,3}(B_2)=K_{5,3}(B_2)=\cdots =\{1\}$.\;
	We again have $d=3$;\; and then
	$e(1)=5$, since $|K_{1,3}(B_2) / K_{2,3}(B_2)|=3^5$,\;
	$e(2)=0$, since $|K_{2,3}(B_2) / K_{3,3}(B_2)|=1=3^0$,\;
	$e(3)=1$, since $|K_{3,3}(B_2) / K_{4,3}(B_2)|=3^1$.\;
	So we have $a=1+2(1\cdot 5 + 2\cdot 0 + 3 \cdot 1)=17$,\;
	$b=(3-1)3=6$.\;
	Again $s(1)=1$. So, the nilpotency class of $A \Wr B_2$ is $17\cdot 1 + (1-1)6=17$.
	
	We got that $A \Wr B_1$ and $A \Wr B_2$ have the same nilpotency class $17$.
Moreover, these wreath products both have the same length of solubility $2$, and the same exponent $27=3 \cdot 3^2$. So, based only on nilpotency class, length of solubility, exponent we cannot yet deduce if the varieties 
$\var{A \Wr B_1}$ and $\var{A \Wr B_2}$ are distinct subvarieties in $\A_3 \A_9$.

However, by Theorem~\ref{Theorem var(Wr) are equal} (in fact, by Lemma~\ref{Lemma case of finie initial factors} already), we have sharper estimates to deduce that $\var{A \Wr B_1} \neq \var{A \Wr B_2}$ (the first factors in which $B_1$ and $B_2$ differ are the initial factors $C_{p^{u_1}}^{\mathfrak m_{p^{u_1}}}=C_{3^2}^2$ for $B_1$, and $C_{p^{v_1}}^{\mathfrak m_{p^{v_1}}}=C_{3^2}$ for $B_2$). 
Repeating some steps of the above proofs for this example, we would get that $A \Wr B_2$ belongs to the variety $\Ni_{3} \B_{9}$ which does not contain the group $A \Wr B_1$.
\end{Example}

\begin{Example} 
	\label{Example D4 Wr C_{3^2}^2}
Let $A_1$ be the Dihedral group $D_4=\langle a, b \mathrel{|} a^4=b^2=1,\; a^b=a^{-1} \rangle$, and let $A_2$ be the Quaternion group  
$Q_8=\langle i, j, k \mathrel{|} i^2=j^2=k^2=ijk \rangle$. These groups are not isomorphic, but they both are of order $8$ and of class $2$, and, moreover, they both generate the variety $\A_2^2 \cap \Ni_2$ \cite{HannaNeumann, Kovacs dihedral}.
As active groups take $B_1=C_{2^2}^3 \times C_2$ and $B_2=C_{2^2}\times C_2^7$. 

Clearly, $\gamma_1(D_4)=D_4$,\;\; 
$\gamma_2(D_4)=D_4'= \langle a^2 \rangle \cong C_2$,\;\; 
$\gamma_3(D_4)=\{1\}$.\;\;
For the group $B_1$ we have 
$K_{1,2}(B_1)=B_1=C_{2^2}^3 \times C_2$,\;
	$K_{2,2}(B_1)=C_{2}^3$,\;
	$K_{3,2}(B_1)=\{1\}$.\;
	Then $d=2$;\; 
	$e(1)=4$, since $|K_{1,2}(B_1) / K_{2,2}(B_1)|=2^4$,\;
	$e(2)=3$, since $|K_{2,2}(B_1) / K_{3,2}(B_1)|=2^3$.\;
	So, we have $a=1+1(1\cdot 4 + 2\cdot 3)=11$,\;
	$b=(2-1)2=2$.\;
	Also, $h=2$ (the class of $A$) and the exponent of $\gamma_1(A)=A$ is $3^1$, and we have $s(1)=1$. Thus, the nilpotency class of $A_1 \Wr B_1$ is 
$\max\{
11\cdot 1 + (2-1)2 \,\, , \,\, 
11\cdot 2 + (1-1)2
\}=22$.
	
$\gamma_1(Q_8)=Q_8$,\;\; 
$\gamma_2(Q_8)=Q_8'= \{\pm 1\} \cong C_2$,\;\; 
$\gamma_3(Q_8)=\{1\}$.\;
Then for $B_2$ we have 
	$K_{1,2}(B_2)=B_2=C_{2^2}\times C_2^7$,\;
	$K_{2,2}(B_2)=C_{2}$,\;
	$K_{3,1}(B_2)=\{1\}$.\;
	Again $d=2$,\; and we have
	$e(1)=8$, since $|K_{1,2}(B_2) / K_{2,2}(B_2)|=2^8$,\;
	$e(2)=1$, since $|K_{2,2}(B_2) / K_{3,2}(B_2)|=2^1$.\;
	Thus, we have $a=1+1(1\cdot 8 + 2\cdot 1)=11$,\;
	$b=(2-1)2=1$.\;
	Using the values $s(1)=2$, $s(2)=1$ again, we get the nilpotency class of $A_2 \Wr B_2$ as 
$\max\{
11\cdot 1 + (2-1)2 \,\, , \,\, 
11\cdot 2 + (1-1)2
\}=22$.

We have that $A_1 \Wr B_1$ and $A_2 \Wr B_2$ have the same nilpotency class $22$. 
And these groups both have the same length of solubility $3$, and the same exponent $16 = 4 \cdot 2^2$. Thus, based only on nilpotency class, length of solubility, exponent we cannot deduce if  
$\var{A_1 \Wr B_1}$ and $\var{A_2 \Wr B_2}$ are distinct subvarieties in $\Ni_3 \A_4$.

But by Theorem~\ref{Theorem var(Wr) are equal} or by Lemma~\ref{Lemma case of finie initial factors}, we have $\var{A_1 \Wr B_1}\neq \var{A_2 \Wr B_2}$. Also, $A_2 \Wr B_2$ does belong to the variety $\Ni_{4} \B_{2}$ which does not contain the group $A_1 \Wr B_1$.
\end{Example}

These examples evolved relatively uncomplicated groups, all of which were finite. But since the criterion of Theorem~\ref{Theorem var(Wr) are equal} is simple for applications, we can easily construct examples with infinite groups also. Say, if in the last example we replace
$B_1=C_{2^2}^3 \times C_2$ by $B_1=C_{2^2}^3 \times C_2^{\aleph_0}$\!\!,\;\; and  $B_2=C_{2^2}\times C_2^7$ by $B_2=C_{2^2}\times C_2^{\aleph_0}$\!\!,\;\; we get wreath products $A_1 \Wr B_1$ and $A_2 \Wr B_2$ which are \textit{non-nilpotent} by theorem of G.Baumslag \cite{Baumslag nilp wr}, and which generate distinct varieties by Theorem~\ref{Theorem var(Wr) are equal}.

\begin{Example} 
Next consider groups which, unlike the ones in previous examples, are not $p$-groups.
Let $A_1=A_2=A=D_4 \times Q_8 \times C_3 \times C_5 \times C_7^{\aleph}$. Take 
$B_1=C_{2^5}^{3} \times C_{2^4}^{\aleph}  \times
C_{2}^{8}  \times
C_3^{\aleph} \times C_7^8$
and
$B_2=C_{2^5}^{3} \times C_{2^4}^{\aleph_0}  \times
C_{2^3}^{2}  \times C_{2}^{9}  \times
C_3^{\aleph_0} \times C_7^9$.
The prime $5$ divides the exponent $m=4\cdot 3 \cdot 5 \cdot 7 = 420$ of $A$ but \textit{not} the exponent $n=32\cdot 3 \cdot 7 = 672$ of $B_1$ and of $B_2$. So, we can ignore the prime $5$ and in Theorem~\ref{Theorem var(Wr) are equal} apply the primes
$p_1=2$,   
$p_2=3$,  
$p_3=7$ only.  

$B_1(p_1)=B_1(2)=C_{2^5}^{3} \times C_{2^4}^{\aleph}  \times
C_{2}^{8}$ and 
$B_2(p_1)=B_2(2)=C_{2^5}^{3} \times C_{2^4}^{\aleph_0}  \times
C_{2^3}^{2}  \times C_{2}^{9}$. Although these $2$-primary components evidently are non-isomorphic, they are equivalent, i.e., $B_1(2) \equiv B_2(2)$. So, for $p_1=2$ the condition of Theorem~\ref{Theorem var(Wr) are equal}  is satisfied, and we can proceed to the next prime.

$B_1(p_2)=B_1(3)=
C_3^{\aleph}$ and 
$B_2(p_2)=B_2(3)=C_3^{\aleph_0} $.
These $3$-primary components again are non-isomorphic, but they are equivalent, i.e., $B_1(3) \equiv B_2(3)$. So, for $p_2=3$ the condition of Theorem~\ref{Theorem var(Wr) are equal} again is satisfied.	

Finally, consider
$B_1(p_3)=B_1(7)=
C_7^8$ and 
$B_2(p_3)=B_2(7)=C_7^9 $.
These $7$-primary components are not only non-isomorphic but \textit{also are non-equivalent}, i.e., $B_1(7) \not\equiv B_2(7)$. So, for $p_3=7$ the condition of Theorem~\ref{Theorem var(Wr) are equal} is \textit{not} satisfied, and we have 	$\var{A_1 \Wr B_1}\neq \var{A_2 \Wr B_2}$.
Notice how the difference of factors $C_7^8$ and $C_7^9$ matters for equality \eqref{Equation varieties are equal}, whereas the difference of factors, say, $C_2^8$ and $C_2^9$ does \textit{not} matter.
\end{Example}

\noindent
\textbf{Acknowledgements:}
The current work is supported  by the joint grant 18RF-109
of RFBR and SCS MES RA, and by the 15T-1A258 grant of SCS MES RA.
I am very much thankful to the referee for pointing out a miscalculation issue in the proof of Lemma~\ref{Lemma case of finie initial factors}.
I also thank my booth universities YSU and AUA for substantial support.

\medskip

%%%%%%%%%%%%%%%%%%%%%%%%%%%%%%%%%%%%%%%%%%%%%%
%%%%%%%%%%%%%%%%%%%%%%%%%%%%%%%%%%%%%%%%%%%%%%
%%%%%%%%%%%%%%%%%%%%%%%%%%%%%%%%%%%%%%%%%%%%%%

{\footnotesize
\vskip3mm
\noindent
%Informatics and Applied Mathematics Department\\
%Yerevan State University\\
%Alex Manoogian 1, Yerevan 0025, Armenia.\\ 
%E-mails: \href{mailto:vmikaelian@ysu.am}{vmikaelian@ysu.am},
E-mail: \href{mailto:v.mikaelian@gmail.com}{v.mikaelian@gmail.com}

\end{document}

Except for boolean algebra, there is no theory more universally employed in mathematics than linear algebra; and there is hardly any theory which is more elementary, in spite of the fact that generations of professors and textbook writers have obscured its simplicity by preposterous calculations with matrices

There is hardly any theory which is more elementary than linear algebra, in spite of the fact that generations of professors and textbook writers have obscured its simplicity by preposterous calculations with matrices.

Jean Dieudonn\'e